\documentclass[a4paper]{amsart}

\usepackage{graphics,amssymb,enumerate}
\usepackage{hyperref}

\newcommand {\bd} {\begin{displaymath}}
\newcommand {\ed} {\end{displaymath}}
\newcommand {\be} {\begin{equation}}
\newcommand {\ee} {\end{equation}}
\newcommand {\bea} {\begin{eqnarray}}
\newcommand {\eea} {\end{eqnarray}}
\newcommand {\no} {\noindent}
\newcommand {\la} {\lambda}
\newcommand {\al} {\alpha}
\newcommand {\pt} {\partial}
\newtheorem{lemma}{Lemma}
\newtheorem{thm}{Theorem}
\newtheorem{pro}{Proposition}
\newtheorem{cor}{Corollary}
\newtheorem{exa}{Example}
\newtheorem{rem}{Remark}
\newtheorem{dfn}{Definition}

\Large
\begin{document}

\title{  Algebraic Integrability of Lotka-Volterra  \\ equations in three dimensions}

\author{ Kyriacos Constandinides   and Pantelis A.~Damianou}
\address{Department of Mathematics and Statistics\\
University of Cyprus\\
P.O.~Box 20537, 1678 Nicosia\\Cyprus}
\email{kyriacos27@hotmail.com, \ \    damianou@ucy.ac.cy}

\begin{abstract}
 We examine the algebraic complete integrability
of Lotka-Volterra equations in three  dimensions. We
restrict our attention to Lotka-Volterra systems  defined
by a skew symmetric matrix. We obtain  a complete classification
of such systems.  The  classification is obtained using
Painlev$\acute{\rm e}$ analysis and more specifically by  the use
of  Kowalevski exponents.  The imposition of certain integrality
conditions on the Kowalevski exponents gives necessary conditions
for the algebraic integrability of the corresponding systems. We also show
that the conditions are  sufficient.
\end{abstract}

\date{19 September  2009}
\maketitle

\section{Introduction}
The Lotka-Volterra model is a basic model of predator-prey
interactions. The model was developed independently by Alfred
Lotka (1925), and Vito Volterra (1926).  It forms the basis for
many models used today in the analysis of population dynamics.  In three dimensions it describes the dynamics of
a biological system where three species interact.

 The most general form of Lotka-Volterra equations is
\be \label{l-v} \dot x_i = \varepsilon_i x_i + \sum_{j=1}^n a_{ij} x_i x_j, \ \
i=1,2, \dots , n \ . \ee

 We consider  Lotka-Volterra equations without linear terms ($\varepsilon_i=0$), and where
the matrix of interaction coefficients  $A= ( a_{ij} ) $ is  skew-symmetric. This is a natural assumption
related to the principle that crowding inhibits growth.
 The special case of Kac-van Moerbeke  system  (KM-system) was used to
describe population evolution in a hierarchical system of
competing individuals. The KM-system has close connection with the
Toda lattice.  The Lotka-Volterra equations were studied by many authors in
its various aspects, e.g. complete integrability \cite{boun1}
Poisson and bi-Hamiltonian formulation (\cite{cron}, \cite{damianou1},
\cite{kern}), stability of solutions and Darboux polynomials
(\cite{cair}, \cite{moul}).

In this paper we examine the algebraic complete integrability
of such  Lotka-Volterra equations in three  dimensions.
The basic
tools for the required classification are, the use of
Painlev\'e  analysis, the examination of the
eigenvalues of the Kowalevski matrix and other standard Lax pair
and Poisson techniques.  The  Kowalevski exponents are useful  in establishing integrability or non-integrability
 of Hamiltonian systems;  see \cite{adler}, \cite{pol},  \cite{bor},   \cite{gor1}, \cite{gor2}, \cite{kozlov},  \cite{yosh}.
 The first step is to impose certain  conditions on the
exponents, i.e., we  require that  all the Kowalevski exponents be
integers. This gives a finite list of  values of the parameters
satisfying such conditions. This  step requires some elementary
number theoretic techniques as is usual with such type of
classification. In the three-dimensional case the general
expressions for the Kowalevski exponents are rational and
therefore the number theoretic analysis is manageable.

 The second step is to check that the leading behavior of the
Laurent series solutions agrees with the weights of the
corresponding homogeneous vector field defining  the dynamical
system. In our case the weights are all equal to one and therefore
we must exclude the possibility that some of the Laurent series
have leading terms with poles of order greater than one. This is a
step usually omitted by some authors due to its complexity, but in
this paper we analyse  this in detail.
To accomplish this step we use old-fashioned Painlev\'e
 Analysis, i.e., Laurent series.
 The application of Painlev$\acute{\rm e}$ analysis and
especially of the ARS algorithm (see \cite{adrian} \cite{ars1}, \cite{ars2},
\cite{boun1}, \cite{boun2}, \cite{gor2}) is useful in
calculating the Laurent solution of a system and check if there
are ($n-1$) free parameters.

  In performing Painlev\'e  analysis we use the fact that the sum of the variables is
always a first integral. Surprisingly the  Painlev\'e
analysis does not reveal any additional cases besides the ones
already found by using the Kowalevski exponents. In this
classification of the algebraic completely integrable
Lotka-Volterra systems  we discover,
as expected, some well known integrable systems like the open and
periodic Kac-van Moerbeke systems.

\no To make sure that our conditions are not only necessary but
also sufficient we verify  that the systems obtained are indeed
algebraically completely integrable by checking the number of free
parameters.  We also have to point out that our classification is
up to isomorphism. In other words, if one system is obtained from
another by an invertible linear change of variables, we do not consider
them as different. Modulo  this identification we obtain only  six
classes of solutions.

The Lotka-Volterra system can be expressed in hamiltonian form as follows:
Define a quadratic Poisson  bracket by the formula

\bd \{ x_i,x_j\} = a_{ij} x_i x_j, \ \ i,j =
1,2, \dots , n  \ . \ed
Then the system can be written in the form
 $\dot x_i = \{x_i, H\}$, where $H=\sum_{i=1}^n  x_i$. The Louville
integrability in the three-dimensional case can be easily established.
In addition to the Hamiltonian function $H$, there exists a second intergral, in fact a Casimir $F$.  The  formula for this Casimir is given afterwards.   We
have to point out that in general algebraic integrability does not imply Liouville
integrability  and vice versa.

In this paper we restrict our attention  to the  three dimensional case. For
$n=3$  the system is defined by the matrix
\begin{eqnarray*}
A=\left(%
\begin{array}{ccc}
  0  &  a & b \\
  -a &  0 & c \\
  -b & -c & 0 \\
\end{array}%
\right) \ ,
\end{eqnarray*}
where $a,b,c$ are  constants. We  use the notation $(a,b,c)$ to denote this system. It turns out that the algebraically integrable Lotka-Volterra systems fall either into two infinite families or four exceptional cases:
\begin{thm}
The Lotka-Volterra equations in three dimensions are algebraically complete integrable if and only if $(a,b,c)$ is in the class of

\smallskip
\noindent
$(l_2)$ \ \ \  $(1,0,1)$

\smallskip
\noindent
$(l_3)$ \ \ \  $(1,-1,1)$

\smallskip
\noindent
$(l_4)$  \ \ \  $(1,-1,2)$

\smallskip
\noindent
$(l_6)$ \ \ \  $(1,-2,3)$

\smallskip
\noindent
$(l_{\lambda})$ \ \ \  $( 1,1, \lambda) $  \ \ $\lambda \in {\bf Z}\setminus{0}$.

\smallskip
\noindent
$(l_0)$  \ \ \   $(1, 1+\mu, \mu )$ \ \ $\mu \in {\bf R} \setminus {0}$.

\end{thm}
 We use the notation $l_j$ to indicate that the system has an invariant of degree $j$ or equivalently that the largest Kovalevski exponent is $j$.

In Section 2 we give the basic definitions of weight-homogeneous vector fields and the Kowalevski matrix. The definition of Kowalevski exponents and  relevant  results follow the recent  book \cite{pol}.
  See also  the  review article of Goriely \cite{gor1}  where one  can find many more
properties of these exponents. In Section 3 we give some related  properties of   the Kowalevski
exponents and a criterion of algebraic complete integrability. In Section 4 we define the three dimensional Lotka-Volterra systems and find necessary conditions for their algebraic integrability by analyzing the corresponding Kowalevski exponents.  In Section 5 we show that our classification is complete. Finally, in Section 6 we exclude any solutions that may exist due to higher order poles.

\section{Basic definitions}

\no We begin by  defining  what is a weight
homogeneous polynomial. We follow the notation from \cite{pol}.

\vspace{5mm}

\begin{dfn}

\noindent A polynomial $f \in$ {\bf C} $ \left[ x_1,x_2, \dots ,
x_n \right]$ is called a \emph{ weight-homogeneous polynomial of
weight $k$} with respect to a vector $v=\left( v_1,v_2, \dots ,
v_n \right)$ if
\begin{displaymath}
f(t^{v_1}x_1, \dots , t^{v_n}x_n)= t^k f(x_1,x_2, \dots , x_n) \ .
\end{displaymath}
The vector $v$ is called the \emph{weight vector}.  The $v_i$ are all positive integers without a commont divisor. The weight $k$
is denoted by $\varpi(f)$.
\end{dfn}

\begin{dfn}

\noindent A polynomial vector field on {\bf C}$^n$,
\begin{eqnarray} \label{vef}
\begin{array}{c}
\dot x_1 = f_1 (x_1,x_2, \dots , x_n) \\
\vdots \\
\dot x_n = f_n (x_1,x_2, \dots , x_n) \\
\end{array}
\end{eqnarray}
is called a  \emph{ weight-homogeneous vector field} of weight $k$
(with respect to a weight vector $v$), if
$\varpi(f_i)=v_i+k=\varpi(x_i)+k$ for $i=1,2, \dots , n$. A
weight-homogeneous vector field of weight 1 is called \emph{
weight-homogeneous vector field}. Furthermore, when all the
weights are equal to 1, this is simply called a \emph{ homogeneous
vector field}.
\end{dfn}

\begin{exa}
 We consider the periodic 5-particle Kac-van Moerbeke lattice
that is given by the quadratic vector field
\begin{eqnarray} \label{vefex}
\dot x_i = x_i (x_{i-1} - x_{i+1}), \ \ \ \ i=1, \dots , 5,
\end{eqnarray}
with $x_i = x_{i+5}$. This system has three independent  constants of motion,
\begin{eqnarray} \label{cons}
\begin{array}{ccl}
F_1 &=& x_1 + x_2 + x_3 + x_4 + x_5, \\
F_2 &=& x_1 x_3 + x_2 x_4 +x_3 x_5 +x_4 x_1 + x_5 x_2, \\
F_3 &=& x_1 x_2 x_3 x_4 x_5. \\
\end{array}
\end{eqnarray}

 Taking $v=(1,1,1,1,1)$, (\ref{vefex}) becomes a
homogeneous vector field and the weights of the integrals of
motion  are $\varpi (F_1)$=1, $ \varpi (F_2) = 2 \
{\rm and} \ \varpi (F_3) = 5.$
\end{exa}

\vspace{3mm}

We now give the definition of an algebraically completely integrable system following
\cite{gor1,gor2}. Note that this definition differs from the one given in \cite{pol}. However, the definition  in  \cite{pol} implies the definition of this paper.
\vspace{3mm}

\begin{dfn}

 A vector field,
\begin{eqnarray*}
\begin{array}{c}
\dot x_1 = f_1 (x_1,x_2, \dots , x_n) \\
\vdots \\
\dot x_n = f_n (x_1,x_2, \dots , x_n), \\
\end{array}
\end{eqnarray*}
is called an  \emph{ algebraically completely integrable system}
(a.c.i.) if its solution can be expressed as Laurent series
\begin{displaymath}
x_i (t)= \frac{1}{t^{v_i}} \sum_{k=0}^{\infty} x_i^{(k)} t^k, \ \
\ \ i=1,2, \dots , n,
\end{displaymath}
\no where $n-1$ of the coefficients $x_i^{(k)}$ are free
parameters.
\end{dfn}

\subsection{Kowalevski Exponents}

\no The following Proposition is important for two reasons. First,  it  gives  an induction formula for finding  the
Laurent solution of a weight-homogeneous vector field and second it  defines  the Kowalevski exponents which is an  important tool for our classification.

\begin{pro} \label{pro}

\noindent Suppose that we have  a weight-homogeneous vector field on
{\rm \bf C}$^n$ given by
\begin{displaymath}
\dot x_i = f_i (x_1, \dots , x_n), \ \ \ \ i=1,2, \dots , n,
\end{displaymath}
and suppose that
\begin{equation} \label{sol}
x_i (t)= \frac{1}{t^{v_i}} \sum_{k=0}^{\infty} x_i^{(k)} t^k, \ \
\ \ i=1,2, \dots , n
\end{equation}
is a weight-homogeneous Laurent solution for this vector field.
Then the leading coefficients, $x_i^{(0)}$, satisfy the non linear
algebraic equations
\begin{eqnarray} \label{indlo}
\begin{array}{c}
v_1 x_1^{(0)} + f_1 (x_1^{(0)}, \dots , x_n^{(0)}) =0, \\
\vdots \\
v_n x_n^{(0)} + f_n (x_1^{(0)}, \dots , x_n^{(0)}) =0, \\
\end{array}
\end{eqnarray}
\noindent while the subsequent terms $x_i^{(k)}$ satisfy
\begin{eqnarray} \label{7.9}
\left( k {\rm Id}_n - {\mathcal{K}} \left( x^{(0)} \right) \right)
x^{(k)} = R^{(k)},
\end{eqnarray}

\noindent where $x^{(k)}= \left( \begin{array}{c}
                                   x_1^{(k)} \\
                                   \vdots \\
                                   x_n^{(k)} \\
                                 \end{array} \right)$
and $R^{(k)}= \left( \begin{array}{c}
                       R_1^{(k)} \\
                       \vdots \\
                       R_n^{(k)} \\
                       \end{array} \right) $.
$R^{(k)}$ is a polynomial, which depends on the variables
$x_1^{(l)}, \dots , x_n^{(l)}$ with $0 \leq l < k$ only. The
elements of the $n \times n$ matrix $\mathcal{K}$ are given by
\begin{eqnarray} \label{kowa}
{\mathcal{K}}_{i,j} := \frac{\partial f_i}{\partial x_j} +v_i
\delta_{ij},
\end{eqnarray}
\noindent where $\delta$ is the Kronecker delta.
\end{pro}

\begin{rem} The pole order $v_i$ of $x_i$ in (\ref{sol}) coincides with  the
$i$th component of the weight vector. The number, $v_i$, is not necessarily the
pole order of $x_i$ because some of the $x_i^{(0)}$ that can be
calculated solving (\ref{indlo}) may be equal to zero.
\end{rem}

\begin{dfn}

\no

\noindent The system (\ref{indlo}) is called the \emph{
indicial equation}  and its solution set is called the \emph{
indicial locus} and it is denoted by ${\mathcal{I}}$. The $n \times n$ matrix
${\mathcal{K}}$, defined by (\ref{kowa}), is called the \emph{
Kowalevski matrix} and its eigenvalues are called \emph { Kowalevski
exponents} (a terminology due to Yoshida).
\end{dfn}

\vspace{3mm}

\no A necessary condition for algebraic integrability is that
$n-1$ eigenvalues of $\mathcal{K}$ should be integers. It turns
out that the last eigenvalue is always $-1$. The eigenvector that
corresponds to $-1$ is also known. We have the following  Proposition which  can be found in \cite{pol}.

\vspace{3mm}

\begin{pro} \label{pro-1}

\no

\noindent For any $m$ which belongs to the indicial locus
${\mathcal{I}}$, except for the trivial element, the Kowalevski
matrix ${\mathcal{K}}(m)$ of a weight homogeneous vector field
always has $-1$ as an eigenvalue. The corresponding eigenspace
contains $(v_1 m_1, \dots , v_n m_n)^T$ as an eigenvector.
\end{pro}

\section{Properties of Kowalevski exponents} \label{sec}

\no In this section we state some properties of  Kowalevski exponents clarifying  the connection with the degrees of the  first integrals. We also   give a
 necessary condition  for a system to be  algebraically  completely integrable. The
following results can be found in \cite{furta, gor1, kozlov, sade,yosh}.

\begin{thm} \label{thm1}

\noindent If the weight-homogeneous system $\dot x = f(x)$ has $k$
independent algebraic first integrals $I_1, \dots , I_k$ of
weighted degrees $d_1, \dots , d_k$ and Kowalevski exponents
$\rho_2, \dots , \rho_n$, then there exists a $k \times (n-1)$
matrix $\mathcal{N}$ with integer entries, such that
\bd
\sum_{j=2}^n {\mathcal{N}}_{ij} \,  \rho_j = d_i, \ \ \ \ \ i = 1,
\dots , k.
\ed
\end{thm}

\vspace{3mm}

\noindent From this theorem we have the two following corollaries:

\vspace{3mm}

\begin{cor} \label{cor1}

\no

\noindent If the Kowalevski exponents are {\bf Z}-independent,
then there is no rational first integrals.
\end{cor}

\begin{cor} \label{cor2}

 If the Kowalevski exponents are {\bf N}-independent,
then there is no polynomial first integrals.
\end{cor}

We also have the following theorem, see \cite{furta, gor1, tsyg}.
\begin{thm} \label{thm3}

 Suppose that the system (\ref{l-v}) possesses a homogeneous first integral $F_m$ of degree $m$. Then there exists a set of non-negative integers $k_2, \dots, k_n$ such that
\bd
\sum_{j=2}^n k_j \rho_j = m  \qquad  k_2+k_3+ \dots +k_n \le m \ .
\ed
\end{thm}

\no The next theorem which can be found in \cite{pol} gives us a
necessary condition for a system to be algebraically integrable. This criterion can be checked easily simply by computing the  Kowalevski exponents.

\begin{thm} \label{thm2}

\no

\noindent Let $\rho_1=-1$.  A necessary condition for a system of the form
(\ref{vef}) to be algebraically completely integrable is that all the
Kowalevski exponents $\rho_2, \dots, \rho_n$  should be  integers.

\end{thm}

\section{Lotka-Volterra systems}

\subsection{ Hamiltonian formulation}   Consider a   Lotka-Volterra system  of the form
\begin{equation} \label{lotka}
\dot x_j = \sum_{k=1}^n a_{jk} x_j x_k, \ \ {\rm for} \ \ j=1,2,
\dots , n,
\end{equation}
\noindent where the matrix $A=(a_{ij})$ is  constant  and skew symmetric.

 There is a sympletic realization of the system which goes back to Volterra.
 In other words a projection
from {\bf R}$^{2n} \mapsto$ {\bf R}$^n$   from a symplectic space to a Poisson space.
 Volterra defined the variables
$$q_i(t)= \int_0^t u_i(s)ds$$ \no and
$$p_i(t)= \ln (\dot q_i)-\frac{1}{2} \sum_{k=1}^n a_{ik} q_k,$$
\no for $i=1,2, \dots , n$.  Now the number of variables is doubled and Volterra's
transformation is given explicitly  by
 $$x_i=e^{p_i+\frac{1}{2} \sum_{k=1}^n a_{ik} q_k} \ \
{\rm for} \ \ i=1,2, \dots, n.$$

 The Hamiltonian in these coordinates becomes
$$H=\sum_{i=1}^n x_i =\sum_{i=1}^n \dot q_i = \sum_{i=1}^n {\rm \bf
e}^{p_i+ \frac{1}{2} \sum_{k=1}^n a_{ik} q_k}.$$

 The equations (\ref{lotka}) can be written in Hamiltonian form  \bd \begin{array}{ccc}
  \dot q_i =& \frac{\pt H}{\pt p_i} & = \{q_i,H\}, \\
  \dot p_i =& -\frac{\pt H}{\pt q_i} & =\{p_i,H\}, \\
\end{array}\ed
 $i=1,2, \dots, n$, and the bracket $\{\cdot,\cdot\}$ is the
standard symplectic bracket on ${\bf R}^{2n}$:
\bd \{q_i,p_j\}= \delta_{ij}=\left\{%
\begin{array}{ll}
    1, & \hbox{if} \ \ i=j\\
    0, & \hbox{if} \ \ i \neq j\\
\end{array}%
\right.  ,  \ \ i,j=1,2, \dots , n \ ; \ed
  all  other brackets are  zero. The corresponding Poisson bracket in
 $x$ coordinates is  quadratic $$\{x_i,x_j\}=a_{ij} x_i x_j, \ \ i,j=1,2,
\dots , n.$$

\vspace{3mm}

\no Equations (\ref{lotka})  in $x$ coordinates  are obtained by using this Poisson
bracket and the Hamiltonian, $H=x_1+x_2+ \dots +x_n$.

\vspace{3mm}

\subsection{The three-dimensional case}

\noindent

In this paper we restrict our attention  to the  three dimensional case. For
$n=3$  the system is defined by the matrix
\begin{eqnarray} \label{lotmat3}
A=\left(%
\begin{array}{ccc}
  0  &  a & b \\
  -a &  0 & c \\
  -b & -c & 0 \\
\end{array}%
\right) \ ,
\end{eqnarray}
where $a,b,c$ are real constants.

\noindent Using equations  (\ref{kowa})  we obtain  the Kowalevski matrix  \large
\begin{eqnarray} \label{kmat1}
\left(%
\begin{array}{ccc}
  ax_2^{(0)}+bx_3^{(0)}+1 & ax_1^{(0)} & bx_1^{(0)} \\
  -ax_2^{(0)} & -ax_1^{(0)}+cx_3^{(0)}+1 & cx_2^{(0)} \\
  -bx_3^{(0)} & -cx_3^{(0)} & -bx_1^{(0)}-cx_2^{(0)}+1 \\
\end{array}%
\right),
\end{eqnarray}
\noindent where $x^{(0)}= \left( x_1^{(0)}, \ x_2^{(0)}, \
x_3^{(0)} \right)$ is an element of the indicial locus,  i.e.,  a
solution of the simultaneous equation (\ref{indlo}), which in this
case is written as
\begin{eqnarray}
\begin{array}{rcl}
x_1^{(0)} + a x_1^{(0)} x_2^{(0)} + b x_1^{(0)} x_3^{(0)} & = & 0, \\
x_2^{(0)} - a x_1^{(0)} x_2^{(0)} + c x_2^{(0)} x_3^{(0)} & = & 0, \\
x_3^{(0)} - b x_1^{(0)} x_3^{(0)} - c x_2^{(0)} x_3^{(0)} & = & 0. \\
\end{array}
\end{eqnarray}
\noindent In Table 1 we list  the corresponding Kowalevski
exponents for each element of
the indicial locus.

\begin{table}[h]
\begin{center}
\begin{tabular}{|c|c||c|c|}
  \hline
  Vector $x^{(0)}$ & Kowalevski & Vector $x^{(0)}$
  & Kowalevski \\
  & exponents & & exponents\\
  \hline
  \cline{1-4}
  & & & \\
  (0,0,0) & 1,1,1 &
  (0,$\frac{1}{c}$,-$\frac{1}{c}$) & -1,1,$\frac{a-b+c}{c}$ \\
  & & & \\
  ($\frac{1}{b}$,0,-$\frac{1}{b}$) & -1,1,-$\frac{a-b+c}{b}$ &
  ($\frac{1}{a}$,-$\frac{1}{a}$,0) & -1,1,$\frac{a-b+c}{a}$ \\
  & & & \\
  \hline
\end{tabular}
\end{center}
\caption{Kowalevski exponents of 3x3 Lotka-Volterra equations}
\end{table}

\noindent A necessary condition of algebraic integrability is that all the   Kowalevski exponents must be integers. So we have to solve the simultaneous Diophantine
equations

\begin{equation} \label{simul}
\frac{a-b+c}{a}=k_1, \ \ \frac{a-b+c}{c}=k_2, \ \
-\frac{a-b+c}{b}=k_3,
\end{equation}

\noindent where $k_1, \ k_2, \ k_3 \in$ {\bf Z}. The case $b=a+c$
for which $k_1=k_2=k_3=0$ is investigated below. Solving
(\ref{simul}) we find that
\begin{eqnarray} \label{case3}
k_3=\frac{k_1 k_2}{k_1 k_2-k_1-k_2}, \ \ \left\{%
\begin{array}{ll}
    c= \frac{k_1}{k_2}a, & \hbox{$b=\frac{k_1+k_2-k_1 k_2}{k_2}a$ } \\
    a= \frac{k_2}{k_1}c, & \hbox{$b=\frac{k_1+k_2-k_1 k_2}{k_1}c $} \\
\end{array}%
\right.
\end{eqnarray}
\begin{eqnarray} \label{case1}
k_2=\frac{k_1 k_3}{k_1 k_3-k_1-k_3}, \ \ \left\{%
\begin{array}{ll}
    b= -\frac{k_1}{k_3}a, & \hbox{$c=\frac{k_1 + k_3-k_1k_3}{k_3}a$} \\
    a= -\frac{k_3}{k_1}b, & \hbox{$c=\frac{k_1 + k_3-k_1k_3}{k_1}b$} \\
\end{array}%
\right.
\end{eqnarray}
\begin{eqnarray} \label{case2}
k_1=\frac{k_2 k_3}{k_2 k_3-k_2-k_3}, \ \ \left\{%
\begin{array}{ll}
    b=-\frac{k_2}{k_3}c, & \hbox{$a= \frac{k_2 k_3-k_2-k_3}{k_3}c$} \\
    c=-\frac{k_3}{k_2}b, & \hbox{$a= \frac{k_2+k_3-k_2 k_3}{k_2}b$} \\
\end{array}%
\right.
\end{eqnarray}

\noindent  We assume first, that the Kowalevski exponents  are not  zero. We
examine the solution
\bd
k_3=\frac{k_1 k_2}{k_1 k_2-k_1-k_2}, \ \ b=\frac{k_1+k_2-k_1
k_2}{k_2}a, \ \ c= \frac{k_1}{k_2}a.
\ed
\no  We determine the values of  $k_1$ and $k_2$ so that the
fraction,
\begin{equation} \label{frac}
k_3=\frac{k_1k_2}{k_1k_2-k_1-k_2},
\end{equation}
\noindent is an integer. We first consider the case $k_1 k_2-k_1-k_2 \not=0$.

\bigskip
\noindent
{\bf Case I}  Assume  positive values for
both $k_1$ and $k_2$.

 Since
\begin{eqnarray*}
\begin{array}{ccl}
\frac{k_1k_2}{k_1k_2-k_1-k_2}=1+ \frac{k_1+k_2}{k_1k_2-k_1-k_2}

\end{array}
\end{eqnarray*}
it is enough to  to solve the Diophantine equation
\bd
\frac{x+y} {xy-x-y} =z
\ed
for $x$, $y$ positive integers and $z  \in {\bf Z}$.

\begin{lemma}
Let $x, y \in {\bf Z}^{+}$ with $x \le y$. Then
\bd \frac{x+y} {xy-x-y} \in {\bf Z} \ed  if and only if  $(x,y)$ is one of the following:
$(1, \lambda)$, $\lambda \in {\bf Z}^{+}$, $(2,3)$, $(2,4)$, $(2,6)$, $(3,3)$, $(3,6)$, $(4,4)$.

\end{lemma}

\begin{proof}
Since
\bd xy-x-y \le x+y \ed
we have
\bd xy \le 2(x+y) \le 4y \ . \ed
Since $y \not=0$  we get $x\le 4$.  Therefore $x=1,2,3,4$.  We examine each case separately.
\begin{itemize}
\item   If $x=1$
\bd \frac{x+y} {xy-x-y} =\frac{1+y}{-1}=-1-y \in {\bf Z} \ . \ed
Therefore $(1 , \lambda)$, $\lambda  \in {\bf Z}^{+}$  is always a  solution.
\item
Suppose $x=2$.   Then
\bd \frac{x+y} {xy-x-y} =\frac{2+y}{y-2} = 1 + \frac{4} {y-2} \ed
should be an integer. Therefore $y-2=\pm 1, \pm 2, \pm 4 $.  We obtain the solutions $(2,3)$, $(2,4)$ and $(2,6)$. \item
Suppose $x=3$. Then
\bd \frac{x+y} {xy-x-y} =\frac{y+3}{2y-3} \ed
should be an integer.   Therefore
\bd 2y-3 \le y+3 \ed
and we obtain $y \le 6$.  We obtain the solutions $(3,3)$ and $(3,6)$.
\item
Suppose $x=4$. Then
\bd \frac{x+y} {xy-x-y} =\frac{y+4}{4y-4} \ed
should be an integer.   Therefore
\bd 3y-4 \le y+4 \ed
and we obtain $y \le 4$.  We obtain the solution $(4,4)$.

\end{itemize}

\end{proof}
Of course, since the fraction
\bd \frac{x+y} {xy-x-y}  \ed
is symmetric with respect to $x$ and $y$, we easily obtain all solutions in positive integers.

We summarize:

\begin{eqnarray}
{\rm For} \ \ 1 \leq k_1 \leq k_2 \ \ \left\{%
\begin{array}{ll}
    k_1=1 \ \ \Longrightarrow & \hbox{$ k_2=\lambda  \in {\bf Z}^{+}$} \\
    k_1=2 \ \ \Longrightarrow & \hbox{$k_2 \in \{ 3,4,6 \}$} \\
    k_1=3 \ \ \Longrightarrow & \hbox{$k_2 \in \{ 3,6 \}$} \\
    k_1=4 \ \ \Longrightarrow & \hbox{$k_2=4$}. \\
\end{array}%
\right.
\end{eqnarray}
Note that the case $k_1=3$, $k_2=3$ implies $k_3=3$ and we obtain the periodic KM-system $(1,-1,1)$.

\bigskip
\noindent
{\bf Case II}

\noindent Suppose  one of them, say $k_1$, is positive while   the other, $k_2$,  is negative. Let  $k_2=-x, \ x>0$. Then
$$ k_3=\frac{-k_1x}{-k_1x-k_1+x}=\frac{k_1x}{k_1x+k_1-x}=1+\frac{x-k_1} { k_1 x +k_1-x} $$
\noindent

It  is enough to  to solve the Diophantine equation
\bd
\frac{x-y} {xy+y-x} =z
\ed
for $x$, $y$ positive integers and $z  \in {\bf Z}$.

\begin{lemma}
Let $x, y \in {\bf Z}^{+}$.  Then
\bd \frac{x-y} {xy+y-x} \in {\bf Z} \ed  if and only if  $(x,y)$ is of the form
$( \lambda,1 )$ or $(\lambda, \lambda)$ with $\lambda \in {\bf Z}^{+}$.

\end{lemma}
\begin{proof}
If $y=1$  then
\bd \frac{x-y} {xy+y-x}=x \in {\bf Z} \ . \ed
Therefore a pair of the form $( \lambda,1 )$ is always a solution.

Assume $y >1$.  We note that
\bd \frac{xy}{xy+y-x}=1 +\frac{x-y} {xy+y-x} \ed
and therefore $xy+y-x \le xy$  implies $y \le x$.
If $x=y$ then our fraction is clearly an integer.   On the other hand,  if  $y <x$, then  the fraction
\bd \frac{x-y} {xy+y-x} \notin {\bf Z} \  \ed
since
\bd (y-1)x+y \ge x+y > x-y \ . \ed
\end{proof}

If $k_1=1$ then $k_2=-\lambda$ and $k_3=\lambda$.  Similarly,  if $k_1=\lambda$ then $k_2=-\lambda$ and $k_3=1$. The two cases are isomorphic and correspond to Case 5 in Table 3.

\bigskip
\noindent
{\bf Case III}

 If we take negative values for both $k_1$ and $k_2$,
then $$k_3=\frac{xy}{xy+x+y}=1-\frac{x+y}{xy+x+y},$$ \noindent
where $k_1=-x$ and $k_2=-y$ with $x,y>0$. We have that $xy>0$ and
$xy+x+y>0$ so that the Kowalevski exponent is an integer if
$$xy+x+y \leq x+y $$     which implies  $ xy \leq 0$, a contradiction. Therefore, in this case $k_3$ cannot be an integer.

 This completes the analysis of the case  $k_1k_2-k_1-k_2 \neq 0$.

 Now suppose $k_1k_2-k_1-k_2 = 0$.

In this case we have $k_1+k_2=k_1 k_2$ and obviously  (since  we assume non-zero Kowalevski exponents) we  must have
$k_1=k_2=2$. We easily obtain $a=c$ and $b=0$. This  system is equivalent to the open KM-system (also known as the Volterra lattice). This is Case 1 in Table 3.

 This concludes our  analysis.  We have obtained necessary conditions for the algebraic integrability of Lotka-Volterra systems in three dimensions and the results are summarized in Table 2. In Table 3 we also include the case of a zero exponent i.e. $b=a+c$. Note that the case $b=a+c$  which is equivalent to $(1, 1+\mu, \mu)$ for $\mu  \in {\bf R}\setminus {0}$ was also considered   in \cite{boun1} from a different point of view.

\subsection{Equivalence}

\noindent

In order to have a more compact classification, we define an equivalence between two Lotka-Volterra systems.
To begin with,  common factors can be removed. In other words,
suppose  that  matrix $A= \left( a_{ij} \right) $ in (\ref{lotka})  has a common factor $a$.  Precisely, if
$$a_{ij}=C_{ij} a, \ \ {\rm where} \ C_{ij} \in {\rm {\bf R}}, \ \
\ i,j=1,2, \dots ,n,$$

\noindent then the Lotka-Volterra system (\ref{lotka}) can be
simplified  to
$$\dot u_i=\sum_{j=1}^n C_{ij} u_i u_j, \ \ \ i=1,2 \dots ,n,$$

\noindent using the transformation \bd
u_i =a  x_i, \ \ \ i=1,2 \dots ,n. \ed

More generally,   we consider two systems to be isomorphic if there exists
an invertible linear transformation mapping one to the other. Special
cases of isomorphic systems are those that are obtained  from a
given system by applying a permutation of the coordinates. Let  $\sigma \in S_n$, and define a transformation
$$X_i \longmapsto x_{\sigma (i)}, \ \ \ i=1,2, \dots ,n.$$

  The transformed system is then considered equivalent to the original system.  We
illustrate with an  example for $n=3$.

\begin{exa} We prove that the system

\bd
\begin{array}{rclcrcl} \label{iso1}
  \dot x_1 & = & ax_1x_2 -\frac{a}{3} x_1x_3 & & \dot x_1 & = & 3x_1x_2 -x_1x_3 \\
  \dot x_2 & = & -ax_1x_2 + \frac{2a}{3} x_2x_3 & \overrightarrow{ \ \ u_i =a  x_i
  \ \ } &
  \dot x_2 & = & -3x_1x_2 + 2x_2x_3 \\
  \dot x_3 & = & \frac{a}{3}x_1x_3 -\frac{2a}{3} x_2x_3 & & \dot x_3 & = & x_1x_3 -2x_2x_3\\
\end{array}
\ed

\noindent is isomorphic to the system

\bd
\begin{array}{rclcrcl} \label{iso2}
  \dot x_1 & = & ax_1x_2 -2a x_1x_3 & & \dot x_1 & = & x_1x_2 -2x_1x_3 \\
  \dot x_2 & = & -ax_1x_2 + 3a x_2x_3 & \overrightarrow{ \ \ u_i =a  x_i
  \ \ } &
  \dot x_2 & = & -x_1x_2 + 3x_2x_3 \\
  \dot x_3 & = & 2ax_1x_3 -3a x_2x_3 & & \dot x_3 & = & 2x_1x_3 -3x_2x_3\\
\end{array}
\ed

\noindent Applying $\sigma = (1 \ 3 \ 2)$  we have that

\bd
\begin{array}{rccc}
  \dot X_1 = \dot x_{\sigma (1)} = \dot x_3 = & x_1x_3 -2x_2x_3 & =
  & X_2X_1 -2 X_3X_1 \\
  \dot X_2 = \dot x_{\sigma (2)} = \dot x_1 = & 3x_1x_2 -x_1x_3 & =
  & 3X_2X_3 -X_2X_1 \\
  \dot X_3 = \dot x_{\sigma (3)} = \dot x_2 = & -3x_1x_2 +2x_2x_3 & =
  & -3X_2X_3 +2 X_3X_1 \nonumber\\
\end{array},
\ed

\noindent which is the second  vector field.
\end{exa}

\begin{exa}
Note that the  system
\bea \begin{array}{rcl}
      \dot x_1 & = & -x_2x_3 \\
      \dot x_2 & = & x_2x_3 \\
      \dot x_3 & = & x_1x_3-x_2x_3-x_3^2 \\
    \end{array} \nonumber \eea
is equivalent to the open KM-system $(1,0,1)$ under the transformation
\bd (x_1, x_2, x_3) \to (x_2+x_3, x_1, x_2) \ed
but it is not a Lotka-Volterra system.

\end{exa}
\no In Table $2$  we display  the different values of $ \left( a, b, c
\right) $ of the solutions (\ref{case3}), (\ref{case1}) and
(\ref{case2}) of the simultaneous equations (\ref{simul}) which ensure integer Kowalevski exponents for the Lotka-Volterra system
 in three dimensions. We also list the elements of the symmetric
group $S_3$ which realize the  isomorphism. Note that $\lambda \in {\bf Z}\setminus{0}$. The final six  non-isomorphic systems are displayed in Table $3$.

\vspace{2mm}

\begin{table}[h]
\begin{center}
\begin{tabular}{|c||c||l|}
  \hline
  Vector ($a, \ b, \ c $) &  Kowalevski & $\ \ \ \ \ \ \sigma$ \\
  & exponents & \\
  \hline
  \cline{1-3}
   $\left(a,\frac{a}{\lambda},\frac{a}{\lambda}\right)$ & $-1,1,1$ & \\
   $\left(a,a,\lambda a\right)$ & $-1,1,\lambda$ &  $\sigma = (1 \ 3)$ \\
   $\left(a,\lambda a,-a\right)$& $-1,1,-\lambda$ & $\sigma = (2 \ 3)$ \\
  \hline
   $\left(a,-\frac{a}{2},\frac{a}{2}\right)$ & $-1,1,2$ & \\
   $\left(a,-a,2a\right)$ & $-1,1,4$ & $\sigma = (1 \ 3)$ \\
   $\left(a,-2a,a\right)$ & & $\sigma = (1 \ 3 \ 2)$ \\
  \hline
   $\left(a,-a,a\right)$ & $-1,1,3$ & \\
  \hline
   $\left(a,-\frac{a}{3},\frac{2a}{3}\right)$ & & $\sigma = (1 \ 3 \ 2)$ \\
   $\left(a,-\frac{2a}{3},\frac{a}{3}\right)$ & $-1,1,2$ & $\sigma = (1 \ 3)$ \\
   $\left(a,-\frac{3a}{2},\frac{a}{2}\right)$ & $-1,1,3$ & $\sigma = (1 \ 2 \ 3)$ \\
   $\left(a,-\frac{a}{2},\frac{3a}{2}\right)$ & $-1,1,6$ & $\sigma = (2 \ 3)$ \\
   $\left(a,-2a,3a\right)$ & & \\
   $\left(a,-3a,2a\right)$ & & $\sigma = (1 \ 2)$ \\
  \hline
   $\left(a,0,a\right)$ & & \\
   $\left(a,-a,0\right)$ & $-1,1,2$ & $\sigma = (1 \ 3 \ 2)$ \\
   $\left(0,b,-b\right)$ & & $\sigma = (1 \ 2 \ 3)$ \\
  \hline
\end{tabular} \nonumber
\end{center}
\caption{Systems  with integer Kowalevski exponents}
\end{table}

\vspace{3mm}

\begin{exa}
\no The periodic KM system (\cite{kac})  in three dimensions is the system
\bd \label{lotka3}
\dot x_i = \sum_{i=1}^3 a_{ij} x_i x_j, \ \ i=1,2,3,
\ed
\noindent where $A$ is the $3 \times 3$ skew-symmetric matrix
\begin{eqnarray*}
A=\left(%
\begin{array}{ccc}
  0 & -1 & 1 \\
  1 & 0 & -1 \\
  -1 & 1 & 0 \\
\end{array}%
\right).
\end{eqnarray*}
\noindent This system is a special case of the system
(\ref{lotka})  where
$\left(a,b,c\right)=\left(-1,1,-1\right)$. This is  Case 2  in Table 3. The Kowalevski exponents of this
system are $-1,1,3$. The system can be written in the Lax-pair
form $\dot L= \left[L,B\right]$, where
\begin{eqnarray*}
L = \left(%
\begin{array}{ccc}
  0 & x_1 & 1 \\
  1 & 0 & x_2 \\
  x_3 & 1 & 0 \\
\end{array}%
\right), \ \ \ B= \left(%
\begin{array}{ccc}
  0 & 0 & x_1x_2 \\
  x_2x_3 & 0 & 0 \\
  0 & x_1x_3 & 0 \\
\end{array}%
\right).
\end{eqnarray*}

\noindent We have the constants of motion $$H_k={\rm
trace}\left(L^k\right), \ \ k =1,2, \dots $$

\noindent The functions
\begin{eqnarray*}
\begin{array}{rcl}
  H_2 & = & x_1 + x_2 + x_3 \\
  H_3 & = & 1 + x_1 x_2 x_3 \\
\end{array}
\end{eqnarray*}

\noindent are independent constants of motion in involution with respect to the Poisson bracket
\begin{eqnarray*}
\pi= \left(%
\begin{array}{ccc}
  0 & -x_1x_2 & x_1x_3 \\
  x_1x_2 & 0 & -x_2x_3 \\
  -x_1x_3 & x_2x_3 & 0 \\
\end{array}%
\right).
\end{eqnarray*}
\noindent We note that the  positive Kowalevski exponents, $1$ and $3$, correspond to the
degrees of the constants of motion.
\end{exa}

\vspace{3mm}

\noindent We have to  point out  that all  Lotka-Volterra  systems  in three dimensions are integrable in the sense of
Liouville since there exist two constants of motion which  are
independent and in involution. The function $$H=x_1+x_2+x_3$$
\noindent is the Hamiltonian for these systems using the quadratic Poisson bracket
\begin{eqnarray*}
\pi= \left(%
\begin{array}{ccc}
  0 & ax_1x_2 & bx_1x_3 \\
  -ax_1x_2 & 0 & cx_2x_3 \\
  -bx_1x_3 & -cx_2x_3 & 0 \\
\end{array}%
\right).
\end{eqnarray*}
\noindent

The equations of motion can be written in Hamiltonian  form  \bd \dot x_i = \{x_i,H\}, \ \ \ i = 1,2,3 \ . \ed

\noindent

 The second constant of motion, independent of $H$ always exists. It is straightforward to check that the function
  \bd F=x_1^c x_2^{-b}x_3^a \ed   is always a Casimir.  Therefore the system is Liouville integrable for any value of $a, b, c$. This is not the case if $n \ge 4$.

\section{Free Parameters}

 We would like to classify the algebraic completely integrable
Lotka-Volterra equations in three dimensions. In order to use
Proposition \ref{pro} we have to assume
Laurent solutions of the form
\begin{equation} \label{laur}
x_i (t)= \frac{1}{t^{v_i}} \sum_{k=0}^{\infty} x_i^{(k)} t^k, \ \
\ \ i=1,2, \dots , n,
\end{equation}
 where $v_i$ are the components of the weight vector $v$ that
makes the vector field
$$
\dot x_i = f_i (x_1, \dots , x_n), \ \ \ \ i=1,2, \dots , n,
$$
\no to be  weight homogeneous. In our case  the weight vector is  $v=\left(1,1,1\right)$.

To make sure that our classification is complete  we must check that each system obtained by imposing these necessary
conditions is indeed  a.c.i.  This  means
that the Laurent series of the solutions $x_1, \ x_2$ and $x_3$
must have $n-1=2$ free parameters. Using the results  in
\cite{pol}, the free parameters appear in a finite number of steps
of calculation. The first thing to do is to substitute
(\ref{laur}) into equations (\ref{lotka}). After that we equate the coefficients of $t^k$.
We have already equated the coefficients of $t^{-v_i-1}$ by
solving the indicial equation to find $x_i^{(0)}$. Then we call
{\it Step m} ($m \in$ {\bf N}) when  we equate the coefficients of
$t^{-v_i-1+m}$ to find $x_i^{(m)}$. According to
\cite{pol} all the free parameters appear in the first $k_p$ {\it
Steps}, where $k_p$ is the largest (positive) Kowalevski exponent
of the system. The calculations are straightforward and we omit the details.
\begin{table}[h]
\begin{center}
\begin{tabular}{|c||c||c||c|}
  \hline
  {\bf Vector} & {\bf Kowalevski} & {\bf Free parameters}& Degree of \\
  & {\bf exponents} & &  invariant \\
  \hline
  \cline{1-4}
   $\left(a,0,a\right)$  & $-1,1,2$ & $x_3^{(1)}, \ \ x_3^{(2)}$ & 2 \\
  \hline
   $\left(a,-a,a\right)$ & $-1,1,3$ & $x_3^{(1)}, \ \ x_3^{(3)}$ &3\\
  \hline
   $\left(a,-\frac{a}{2},\frac{a}{2}\right)$ & $-1,1,2$
   & $x_3^{(1)}, \ \ x_3^{(2)}$ & 4\\
   & $-1,1,4$ &  & \\
  \hline
   & $-1,1,2$ &  &    \\
   $\left(a,-2a,3a\right)$ & $-1,1,3$ & $x_3^{(1)}, \ \ x_3^{(2)}$ &6\\
   & $-1,1,6$ & & \\
    \hline
   & $-1,1,1$ &  &\\
   $\left(a,\frac{a}{\lambda},\frac{a}{\lambda}\right)$ &
   $-1,1,\lambda$ & $x_1^{(1)}, \ \ x_2^{(1)}$ & $\lambda$ \\
   & $-1,1,-\lambda$ &  &  \\
   \hline
   $\left(a,a+c,c\right)$ & $-1,1,0$ & $x_1^{(1)}, \ \ x_3^{(0)}$   & 0 \\
  \hline
\end{tabular}
\end{center}
\caption{Free parameters for the algebraic complete
integrability of each system}
\end{table}

\vspace{3mm}

\no All the systems that we have obtained turn out to be a.c.i. We summarize the results  in Table 3
where we display the 2 free parameters in each case. We note that the six Cases of Theorem 1 are non-isomorphic by examining the degree of the Casimir.

\section{Higher Order Poles}

In our classification, using the Kowalevski exponents, we assume that the order of the poles agrees with the components of the weight vector, in our case all equal to $1$.  We have to exclude the possibility of missing some cases due to solutions with higher order poles. We show that no such new cases appear.

 Suppose that the Laurent solution of the system is

\bea
\begin{array}{c}
  x_1 (t)= \frac{1}{t^{\nu_1}} \sum_{k=0}^{\infty} x_1^{(k)} t^k,
  \ {\rm with} \ x_1^{(0)} \neq 0, \\
  x_2 (t)= \frac{1}{t^{\nu_2}} \sum_{k=0}^{\infty} x_2^{(k)} t^k,
  \ {\rm with} \ x_2^{(0)} \neq 0, \\
  x_3 (t)= \frac{1}{t^{\nu_3}} \sum_{k=0}^{\infty} x_3^{(k)} t^k,
  \ {\rm with} \ x_3^{(0)} \neq 0. \\
\end{array}
\eea

 If $\nu_1,\nu_2,\nu_3 \leq 1$, then these systems have been already
investigated using Proposition \ref{pro}.  On the other hand, keeping in mind that
$H=x_1+x_2+x_3$ is always a constant of motion, we end-up with  the
following four  cases to consider:

\begin{itemize}
    \item[(i)] $\nu_1=\nu_2=\nu >1$ and $\nu_3<\nu$, or
    \item[(ii)] $\nu_1=\nu_3=\nu > 1$ and $\nu_2<\nu$, or
    \item[(iii)] $\nu_2=\nu_3=\nu > 1$ and $\nu_1<\nu$, or
    \item[(iv)] $\nu_1=\nu_2=\nu_3=\nu > 1$.
\end{itemize}
 Recall that   equations (\ref{lotka}) in three dimensions are:

\be \label{lot31} \dot x_1 = \ \ ax_1x_2+bx_1x_3, \ee \be
\label{lot32} \dot x_2=-ax_1x_2+cx_2x_3, \ee \be \label{lot33}
\dot x_3=-bx_1x_3-cx_2x_3. \ee
 We examine each of the four cases:

\begin{itemize}
    \item[($i$)] {\Large {\fbox{$\nu_1=\nu_2=\nu > 1$ and $\nu_3<\nu$}}}

    \vspace{3mm}

     Since $\nu_1=\nu_2=\nu$ and   using the fact that $H=x_1+x_2+x_3$ is a constant of motion we have that $$x_1^{(0)}=-x_2^{(0)}=\al\neq 0 \ .$$

    \vspace{3mm}
     We also note that $\nu+\nu_3<2\nu$ and $x_1^{(0)}x_2^{(0)} \neq 0$. Equating
    the coefficients of $t^{2\nu}$ of the LHS and RHS of ($\ref{lot31}$) or
    ($\ref{lot32}$), we are led to $a\ x_1^{(0)}x_2^{(0)} = 0$. Therefore $a=0$.

    \vspace{3mm}
     As we know that $\nu_3+1<\nu+\nu_3$, the coefficient of
    $t^{\nu+\nu_3}$ of the RHS of ($\ref{lot33}$) must be equal to
    zero. So $$x_3^{(0)}\left(-b x_1^{(0)} -c x_2^{(0)}\right)=0,$$

     but $x_3^{(0)} \neq 0$ and $x_2^{(0)}=-x_1^{(0)}\neq 0$;
    therefore $b=c$.

    If $b=0$, then from (\ref{lot31}) and
    (\ref{lot32}) we have that $$\dot x_1 = \dot x_2 =0\Longrightarrow x_1,
    \ x_2 {\rm \ are \ constant \ functions},$$

     that is a contradiction because $\nu_1=\nu_2=\nu>1$.

    \vspace{3mm}
     If $b$ and $c$ are non-zero, then the equations
    (\ref{lot31}) and (\ref{lot32}) become
    \be \label{lot34} \dot x_1=bx_1x_3, \ee
    \be \label{lot35} \dot x_2=bx_2x_3. \ee

     Using equation (\ref{lot31}) we obtain  \bd \nu+1=\nu+\nu_3
     \ed
     therefore $\nu_3=1$,
     since $x_1^{(0)}x_3^{(0)} \neq 0$ and $x_2^{(0)}x_3^{(0)} \neq 0$.

    It follows from  (\ref{lot34}) and (\ref{lot35})  that
    $$\frac{\ \dot x_1 \ }{x_1}=\frac{\ \dot x_2 \ }{x_2}=bx_3 \Longrightarrow
    x_1= \kappa x_2, \ \ \kappa \ {\rm is \ a \ constant}. $$

     However, we know that
    $$x_1^{(0)}= - x_2^{(0)} \Longrightarrow \kappa =-1 \Rightarrow
    x_1 =-x_2. $$
     Equation (\ref{lot33}) becomes
    $$\dot x_3 = -b (-x_2) x_3 - bx_2 x_3 = 0 \Longrightarrow x_3
    = c, \ c \ {\rm is \ a \ constant}.$$
    This is a contradiction since
    $\nu_3=1$ and $x_3^{(0)} \neq 0$.

    \vspace{3mm}

    \item[($ii$)] {\Large {\fbox{$\nu_1=\nu_3=\nu > 1$ and $\nu_2<\nu$}}}

    \vspace{3mm}
     It leads to a contradiction, as in  case ($i$).

    \vspace{3mm}
    \item[($iii$)] {\Large {\fbox{$\nu_2=\nu_3=\nu > 1$ and $\nu_1<\nu$}}}

    \vspace{3mm}
     It leads to a contradiction, as in  case ($i$).

    \vspace{3mm}
    \item[($iv$)] {\Large {\fbox{$\nu_1=\nu_2=\nu_3=\nu > 1$}}}

    \vspace{3mm}
     In this case, for $i=1,2,3$,
    $$x_i (t) = \frac{1}{t^\nu} \sum_{k=0}^{\infty} x_i^{(k)} t^k,$$
     we have that the degrees of the leading term of the LHS of the equations
    (\ref{lot31}), (\ref{lot32}) and (\ref{lot33}) are equal to $\nu+1$, but the
    degrees of the leading term RHS of these equations are equal
    to $2\nu$ and so the coefficients of $\frac{1}{t^{\nu+k}}$ of the RHS of
    these equations must be zero for $k=2,3, \dots , \nu$.

    The coefficients of $\frac{1}{t^{\nu+k}}$, $k=1,2, \dots ,
    \nu$, are given by the sums
    \begin{equation} \label{sik}
    S_{i,k}=\sum_{\la =0}^{\nu-k} x_i^{(\la)} u_{i,k}^{(\la)}, \
    {\rm for} \ i=1,2,3,
    \end{equation}
    where
    \begin{eqnarray*}
    \begin{array}{rcrlll}
      u_{1,k}^{(\la)}&=& ax_2^{(\nu-k- \la)} &+& bx_3^{(\nu-k- \la)}, \\
      u_{2,k}^{(\la)}&=& -ax_1^{(\nu-k- \la)} &+& cx_3^{(\nu-k- \la)}, \\
      u_{3,k}^{(\la)} &=& -bx_2^{(\nu-k- \la)} &-& cx_3^{(\nu-k- \la)}. \\
    \end{array}
    \end{eqnarray*}

     Note that
    \begin{equation} \label{note}
      u_{i,k}^{(\la)} = u_{i,j}^{(m)}, \ {\rm if} \ \ k+\la = j+m.
    \end{equation}

 In addition $$S_{i,k}=0, \ {\rm for} \ i=1,2,3 \ {\rm and} \
    k=2,3, \dots ,\nu.$$

     For $k=n$ sum (\ref{sik}) becomes
    $$S_{i,\nu}=x_i^{(0)} u_{i,\nu}^{(0)}=0 \Longrightarrow u_{i,\nu}^{(0)}=0$$
     since $x_i^{(0)} \neq 0$.

     For $k=\nu-1$ we have that
    \begin{eqnarray*}
    \begin{array}{rcl}
      S_{i,\nu-1}&=& x_i^{(0)} u_{i,\nu-1}^{(0)} + x_i^{(1)} u_{i,\nu-1}^{(1)}=0 \\
      (\ref{note}) &\Rightarrow& x_i^{(0)} u_{i,\nu-1}^{(0)} + x_i^{(1)}
      u_{i,\nu}^{(0)} = x_i^{(0)} u_{i,\nu-1}^{(0)}=0 \\
      &\Rightarrow& u_{i,\nu-1}^{(0)}=0 \ {\rm because} \ \ x_i^{(0)} \neq 0. \\
    \end{array}
    \end{eqnarray*}

     Let $m \in \{ 1,2, \dots , \nu-1 \}$ and assume that
    $u_{i,k}^{(0)}=0$ for $k>m$.

     For $k=m$ we have that
    $$S_{i,m} = \sum_{\la =0}^{\nu-m} x_i^{(\la)} u_{i,m}^{(\la)} = x_i^{(0)}
    u_{i,m}^{(0)} + \sum_{\la =1}^{\nu-m} x_i^{(\la)} u_{i,m}^{(\la)}$$
    $$ \ \ \ \ = x_i^{(0)} u_{i,m}^{(0)} + \sum_{\la =1}^{\nu-m} x_i^{(\la)}
    u_{i,m+ \la}^{(0)}=x_i^{(0)} u_{i,m}^{(0)}.$$

     Since  $S_{i,m}=0$ for $m>1$ and, since $x_i^{(0)} \neq
    0$, then $u_{i,m}^{(0)}=0$.

     Now we equate the coefficients of $\frac{1}{t^{\nu+1}}$ on
    both sides of the equations (\ref{lot31})-(\ref{lot33}) to obtain

    \bd S_{i,1}=x_i^{(0)} u_{i,1}^{(0)}=-\nu x_i^{(0)}\ . \ed

    Therefore, $    \nu + u_{i,1}^{(0)}=0 $.

     Therefore we have that
    \begin{eqnarray}
    \begin{array}{rcrlll}
      ax_2^{(\nu-1)}  &+& bx_3^{(\nu-1)} &=& -\nu, \\
      -ax_1^{(\nu-1)} &+& cx_3^{(\nu-1)} &=& -\nu, \\
      -bx_2^{(\nu-1)} &-& cx_3^{(\nu-1)} &=& -\nu. \\
    \end{array}
    \end{eqnarray}

     These simultaneous equations have solutions only if $$b=a+c.$$

     If $a=0 $,  Then $  b=c$ (obviously $b=c\neq 0$). Then the
    system is isomorphic to the following $(0,1,1)$ system:
    \bea \begin{array}{rcl}
      \dot x_1 & = & x_1x_3 \\
      \dot x_2 & = & x_2x_3 \\
      \dot x_3 & = & -x_1x_3-x_2x_3 \\
    \end{array} \nonumber \eea
     Equating the coefficients of $t^{-2\nu}$ ($\nu>1$) in the
    first and second equations we have that $$x_1^{(0)}x_3^{(0)}=x_2^{(0)}
    x_3^{(0)}=0 \ .$$
     This  is impossible because $x_i^{(0)} \neq 0$, for
    $i=1,2,3$.

    \vspace{3mm}

     The same happens if $bc =0$. So in the following
    calculations we assume that $abc \neq 0$.

    \vspace{3mm}

     We will show that there exists no such   solution with
    $\nu \ge 2$. Since $b=a+c$ and the function $H=x_1+x_2+x_3$ is
    a constant of motion, the Lotka-Volterra equations in three
    dimensions can be written in the form  \bea \label{odesys} \begin{array}{rcl}
      \dot x_1 & = & \ \ akx_1 - ax_1^2 +cx_1x_3, \\
      \dot x_2 & = & -\dot x_1-\dot x_3, \\
      \dot x_3 & = & -ckx_3 + cx_3^2 -a x_1x_3, \\
      \end{array}\eea \no where $k$ is the constant value of the
      function $H$. It is straightforward to see  that if $k \neq 0$, then the solution is
    $$ x_1=\frac{kC_1{\rm \bf
e}^{akt}}{C_1{\rm \bf e}^{akt}+a{\rm \bf e}^{-ckt}-C_2}, \
    x_3 = \frac{ka{\rm \bf
e}^{-ckt}}{C_1{\rm \bf e}^{akt}+a{\rm \bf e}^{-ckt}-C_2},$$ \be
    \label{solution}x_2 =k-x_1-x_3=-\frac{kC_2}{C_1{\rm \bf
e}^{akt}+a{\rm \bf e}^{-ckt}
    -C_2}.\ee \no Obviously $C_2 \neq 0$. The pole  $t_*$ satisfies $$C_1{\rm \bf
e}^{akt_*}+a{\rm \bf e}^{-ckt_*}-C_2=0\Rightarrow C_2=C_1{\rm \bf
e}^{akt_*}+a{\rm \bf e}^{-ckt_*} \neq 0$$ \no Hence using De l'
H$\hat{\rm o}$pital Rule we are led to the fact that
$$\lim_{t\rightarrow t_*} (t-t_*)x_2
(t) = \frac{C_2}{aC_1{\rm \bf e}^{akt_*}-ac{\rm \bf
e}^{-ckt_*}}.$$ \no Since the pole order is greater than 1, we
have that
$$\lim_{t\rightarrow t_*} (t-t_*)x_2 (t) = \infty .$$\no Therefore
$$C_1=c {\rm \bf e}^{-(a+c)kt}=c {\rm \bf e}^{-bkt}$$ \no The solution (\ref{solution})
    possesses only one arbitrary constant $k$, but we need
    $2$.

 Now if $k=0$ the solutions of (\ref{odesys}) are

     \bd  x_3 (t)=0, \ \
    x_1 (t) = \frac{1}{at +C_1} \ , \ed
     or \bd   x_1 (t)= \frac{C_1 -c}
    {a(C_1 t + C_2)}, \ \ x_3 (t)=- \frac{1}{C_1 t + C_2} \ . \ed

Both solutions lead  to a contradiction since the pole     order of $x_1$ and $x_3$ is assumed to be  greater than 1.

     Therefore the case $\nu_1=\nu_2=\nu_3=\nu>1$ does not  give us
    any new  algebraically  integrable systems. The conclusion is  that the case $b=a+c$ is algebraically
     integrable  only  when $\nu_1=\nu_2=\nu_3=1$.
\end{itemize}

\smallskip
\noindent
{\bf Acknowledgments} We thank Tassos Bountis for introducing us to this area of research during a short course he gave at the University of Cyprus. We thank Pol Vanhaecke for useful discussions and for pointing out the need to exclude higher order poles.

\end{document}